\documentclass[12pt]{article}
\usepackage{amssymb}
\usepackage{amsthm}
\usepackage{amsmath}
\usepackage{graphicx}
\usepackage{enumerate}
\usepackage{pict2e}
\usepackage{framed}

\DeclareMathOperator{\Max}{Max}

\newtheorem{theorem}{Theorem}

\newtheorem{lemma}[theorem]{Lemma}

\newtheorem{remark}[theorem]{Remark}
\newtheorem{example}[theorem]{Example}

\title{A Cayley theorem for posets}
\author{Ivan~Chajda and Helmut~L\"anger}
\date{}

\begin{document}
	
\footnotetext[1]{Corresponding author: Helmut L\"anger}	

\footnotetext[2]{Support of the research of the first author by the Czech Science Foundation (GA\v CR), project 25-20013L, and by IGA, project P\v rF~2026~009, and support of the research of the second author by the Austrian Science Fund (FWF), project 10.55776/PIN5424624, is gratefully acknowledged.}

\maketitle
	
\begin{abstract}
We show that every poset $\mathbf P=(P,\le)$ satisfying the Ascending Chain Condition can be isomorphically embedded into the poset of all mappings from $P$ to the set $A(\mathbf P)$ of all antichains of $\mathbf P$ equipped with a certain partial order relation. This isomorphism is presented explicitly.
\end{abstract}



The Cayley theorem is familiarly known to everybody working with groups. It states that every group $(G,\cdot)$ is isomorphic to a subgroup of the group of all permutations on the set $G$. In 1963 C.~Holland \cite H extended this result to lattice-ordered groups as follows: Every lattice-ordered group is isomorphic to an $l$-subgroup of the group of automorphisms of a chain.

A number of authors were interested in Cayley-like theorems for various types of algebras. For Boolean algebras, see \cite{BEM} and for Stone algebras \cite{BE} and \cite H. For algebras connected with non-classical logic based on a quasiordered modification of Boolean algebras it was done by the first author \cite C, and for algebras with binary and nullary operations the result was settled by the authors in \cite{CLb}. Z.~\'Esik presented a modification of the Cayley theorem for ternary algebras in \cite E. For distributive lattices the authors presented a corresponding result in \cite{CLa}.

The aim of this short note is to show that a Cayley-like representation can be developed also for posets.

Let $\mathbf P=(P,\le)$ be a poset. On $2^P$ we introduce the following binary relation $\le$:
\[
B\le C\text{ if for every }x\in B\text{ there exists some }y\in C\text{ with }x\le y
\]
for any subsets $B,C$ of $P$. It is easy to see that if $\mathbf P$ is an antichain then every subset of $P$ is an antichain and $(2^P,\le)=(2^P,\subseteq)$ is a poset. However, if $\mathbf P$ is not an antichain then it contains at least two different comparable elements, say $a<b$, and we have $\{b\}\le\{a,b\}$ and $\{a,b\}\le\{b\}$ showing that $(2^P,\le)$ is not a poset in this case. It turns out that if we restrict the relation $\le$ from the set $2^P$ to the subset $A(\mathbf P)$ of $2^P$ consisting of all antichains of $\mathbf P$ then we can prove the following result:

\begin{lemma}
$\big(A(\mathbf P),\le\big)$ is a poset.
\end{lemma}

\begin{proof}
It is easy to see that $\le$ is reflexive and transitive. Hence it suffices to prove antisymmetry of $\le$. Let $B,C\in A(\mathbf P)$ and assume $B\le C$ and $C\le B$. Suppose $a\in B$. Then, because of $B\le C$, there exists some $b\in C$ with $a\le b$. Since $C\le B$, there exists some $c\in B$ with $b\le c$. Because of $a\le b\le c$ we have $a\le c$. But $a,c\in B$ and $B$ is an antichain of $\mathbf P$. So we conclude $a=c$ and hence $a=b\in C$. This shows $B\subseteq C$. From symmetry reasons we obtain $C\subseteq B$ from which we conclude $B=C$. Thus the binary relation $\le$ on $A(\mathbf P)$ is also antisymmetric and hence a partial order relation.
\end{proof}

We are going to show that the poset $(P,\le)$ can be represented by means of mappings from $P$ to $A(\mathbf P)$.

Let $\mathbf P=(P,\le)$ again be a poset, $a,b\in P$ and $B\subseteq P$ and let $\Max B$ denote the set of all maximal elements of $B$. Observe that $\Max B\in A(\mathbf P)$. We say that $\mathbf P$ satisfies the {\em Ascending Chain Condition} if $\mathbf P$ has no infinite ascending chains. Of course, every finite poset satisfies the Ascending Chain Condition. If $\mathbf P$ satisfies the Ascending Chain Condition and $a\in B$ then there exists some $c\in\Max B$ with $a\le c$ and hence $\Max B\ne\emptyset$. The {\em lower cone} $L(a,b)$ of the elements $a$ and $b$ is defined by
\[
L(a,b):=\{x\in P\mid x\le a\text{ and }x\le b\}.
\]
Further we define a mapping $f_a$ from $P$ to $A(\mathbf P)$ as follows:
\[
f_a(x):=\Max L(a,x)\text{ for all }x\in P.
\]

Now we can prove our main result:

\begin{theorem}\label{th1}
{\rm(}{\bf Cayley theorem for posets}{\rm)} Let $\mathbf P$ be a poset satisfying the Ascending Chain Condition. Then $\mathbf P$ can be isomorphically embedded into the poset $\Big(\big(A(\mathbf P)\big)^P,\le\Big)$ by means of the isomorphism $a\mapsto f_a$.
\end{theorem}

\begin{proof}
Let $\mathbf P=(P,\le)$ and $a,b,c\in P$. We have only to show that $a\le b$ if and only if $f_a\le f_b$. If $a\le b$ and $d\in f_a(c)=\Max L(a,c)$ then $d\in L(a,c)$ and hence $d\in L(b,c)$. Since $\mathbf P$ satisfies the Ascending Chain Condition there exists some $e\in\Max L(b,c)=f_b(c)$ with $d\le e$. This shows $f_a(c)\le f_b(c)$. Since $c$ was an arbitrary element of $P$ we conclude $f_a\le f_b$. Conversely, assume $f_a\le f_b$. Then $a\in \Max L(a,a)=f_a(a)\le f_b(a)=\Max L(b,a)$ and hence there exists some $g\in\Max L(b,a)$ with $a\le g$. Now $a\le g\le b$ showing $a\le b$.
\end{proof}

It should be mentioned that another representation for distributive lattices was developed by the authors in \cite{CLa}.

\begin{example}
Consider the poset $\mathbf P$ visualized in Fig.~1:	

\vspace*{-3mm}

\begin{center}
\setlength{\unitlength}{7mm}
\begin{picture}(6,8)
\put(3,1){\circle*{.3}}
\put(1,3){\circle*{.3}}
\put(5,3){\circle*{.3}}
\put(1,5){\circle*{.3}}
\put(5,5){\circle*{.3}}
\put(3,7){\circle*{.3}}
\put(3,1){\line(-1,1)2}
\put(3,1){\line(1,1)2}
\put(3,7){\line(-1,-1)2}
\put(3,7){\line(1,-1)2}
\put(1,3){\line(0,1)2}
\put(5,3){\line(0,1)2}
\put(1,3){\line(2,1)4}
\put(5,3){\line(-2,1)4}
\put(2.85,.3){$0$}
\put(.35,2.85){$a$}
\put(5.4,2.85){$b$}
\put(.35,4.85){$c$}
\put(5.4,4.85){$d$}
\put(2.85,7.4){$1$}
\put(-.15,-.75){{\rm Figure~1. Finite poset $\mathbf P$}}
\end{picture}
\end{center}

\vspace*{4mm}

Since $\mathbf P$ is finite, it satisfies the Ascending Chain Condition. For every $z\in P$ we compute the mapping $f_z$:
\[
\begin{array}{c|c|c|c|c|c|c}
x & f_0(x) & f_a(x) & f_b(x) & f_c(x)  & f_d(x)  & f_1(x) \\
\hline
0 & \{0\}  & \{0\}  & \{0\}  &  \{0\}  &  \{0\}  & \{0\} \\
a & \{0\}  & \{a\}  & \{0\}  &  \{a\}  &  \{a\}  & \{a\} \\
b & \{0\}  & \{0\}  & \{b\}  &  \{b\}  &  \{b\}  & \{b\} \\
c & \{0\}  & \{a\}  & \{b\}  &  \{c\}  & \{a,b\} & \{c\} \\
d & \{0\}  & \{a\}  & \{b\}  & \{a,b\} &  \{d\}  & \{d\} \\
1 & \{0\}  & \{a\}  & \{b\}  &  \{c\}  &  \{d\}  & \{1\}
\end{array}
\]
Then $\mathbf P$ is isomorphic to the poset depicted in Fig.~2:

\vspace*{-3mm}

\begin{center}
\setlength{\unitlength}{7mm}
\begin{picture}(6,8)
\put(3,1){\circle*{.3}}
\put(1,3){\circle*{.3}}
\put(5,3){\circle*{.3}}
\put(1,5){\circle*{.3}}
\put(5,5){\circle*{.3}}
\put(3,7){\circle*{.3}}
\put(3,1){\line(-1,1)2}
\put(3,1){\line(1,1)2}
\put(3,7){\line(-1,-1)2}
\put(3,7){\line(1,-1)2}
\put(1,3){\line(0,1)2}
\put(5,3){\line(0,1)2}
\put(1,3){\line(2,1)4}
\put(5,3){\line(-2,1)4}
\put(2.75,.3){$f_0$}
\put(.15,2.85){$f_a$}
\put(5.3,2.85){$f_b$}
\put(.15,4.85){$f_c$}
\put(5.3,4.85){$f_d$}
\put(2.75,7.4){$f_1$}
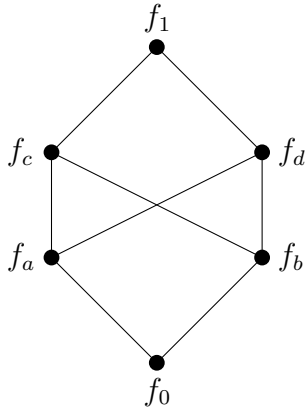
\put(-1.2,-.75){{\rm Figure~2. Poset isomorphic to $\mathbf P$}}
\end{picture}
\end{center}

\vspace*{4mm}

which is a subposet of the poset $\Big(\big(A(\mathbf P)\big)^P,\le\Big)$.

The poset $\big(A(\mathbf P),\le\big)$ of all antichains of $\mathbf P$ is visualized in Fig,~3:

\vspace*{-7mm}

\begin{center}
\setlength{\unitlength}{7mm}
\begin{picture}(6,14)
\put(3,1){\circle*{.3}}
\put(3,3){\circle*{.3}}
\put(1,5){\circle*{.3}}
\put(5,5){\circle*{.3}}
\put(3,7){\circle*{.3}}
\put(1,9){\circle*{.3}}
\put(5,9){\circle*{.3}}
\put(3,11){\circle*{.3}}
\put(3,13){\circle*{.3}}
\put(3,3){\line(-1,1)2}
\put(3,3){\line(1,1)2}
\put(3,3){\line(0,-1)2}
\put(3,11){\line(-1,-1)2}
\put(3,11){\line(0,1)2}
\put(3,11){\line(1,-1)2}
\put(1,5){\line(1,1)4}
\put(5,5){\line(-1,1)4}
\put(2.9,.3){$\emptyset$}
\put(3.4,2.85){$\{0\}$}
\put(5.3,4.85){$\{b\}$}
\put(3.4,6.85){$\{a,b\}$}
\put(5.3,8.85){$\{d\}$}
\put(3.4,10.85){$\{c,d\}$}
\put(-.2,4.85){$\{a\}$}
\put(-.2,8.85){$\{c\}$}
\put(2.5,13.4){$\{1\}$}
\put(-.9,-.75){{\rm Figure~3. The poset $\big(A(\mathbf P),\le\big)$}}
\end{picture}
\end{center}

\vspace*{4mm}

\end{example}

\begin{remark}
Assume that the poset $\mathbf P=(P,\le)$ is a lattice $(P,\vee,\wedge)$. Then $\Max L(a,x)=a\wedge x$ for all $a,x\in P$. Theorem~\ref{th1} ensures that the lattice $\mathbf P$ is isomorphic as a poset to subposet $(\{f_a\mid a\in P\},\le)$ of the poset $\Big(\big(A(\mathbf P)\big)^P,\le\Big)$ and hence $(\{f_a\mid a\in P\},\le)$ itself is a lattice. If $\big(A(\mathbf P),\le\big)$ is a lattice $\big(A(\mathbf P),\vee,\wedge\big)$ {\rm(}which is equivalent to $\Big(\big(A(\mathbf P)\big)^P,\le\Big)$ being a lattice{\rm)} then the mapping $a\mapsto f_a$ need not be a homomorphism from the lattice $(P,\vee,\wedge)$ to the lattice $\Big(\big(A(\mathbf P)\big)^P,\vee,\wedge\Big)$ as the following example shows.
\end{remark}

\begin{example}
Consider the modular lattice $\mathbf M_3=(M_3,\vee,\wedge)$ depicted in Fig.~4:

\vspace*{-3mm}

\begin{center}
\setlength{\unitlength}{7mm}
\begin{picture}(6,6)
\put(3,1){\circle*{.3}}
\put(1,3){\circle*{.3}}
\put(5,3){\circle*{.3}}
\put(3,3){\circle*{.3}}
\put(3,5){\circle*{.3}}
\put(3,1){\line(-1,1)2}
\put(3,1){\line(1,1)2}
\put(3,5){\line(-1,-1)2}
\put(3,5){\line(1,-1)2}
\put(3,1){\line(0,1)4}
\put(2.85,.3){$0$}
\put(.35,2.85){$a$}
\put(5.4,2.85){$c$}
\put(3.4,2.85){$b$}
\put(2.85,5.4){$1$}
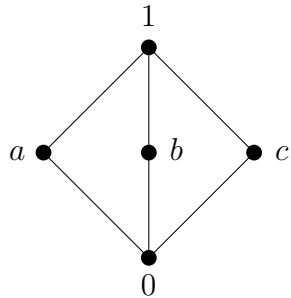
\put(-.9,-.75){{\rm Figure~4. Modular lattice $\mathbf M_3$}}
\end{picture}
\end{center}

\vspace*{4mm}

The poset $\big(A(\mathbf M_3),\le\big)$ is visualized in Fig.~5:

\vspace*{-7mm}

\begin{center}
\setlength{\unitlength}{7mm}
\begin{picture}(6,12)
\put(3,1){\circle*{.3}}
\put(3,3){\circle*{.3}}
\put(1,5){\circle*{.3}}
\put(3,5){\circle*{.3}}
\put(5,5){\circle*{.3}}
\put(1,7){\circle*{.3}}
\put(3,7){\circle*{.3}}
\put(5,7){\circle*{.3}}
\put(3,9){\circle*{.3}}
\put(3,11){\circle*{.3}}
\put(3,1){\line(0,1)4}
\put(3,3){\line(-1,1)2}
\put(3,3){\line(1,1)2}
\put(3,5){\line(-1,1)2}
\put(3,5){\line(1,1)2}
\put(3,7){\line(-1,-1)2}
\put(3,7){\line(1,-1)2}
\put(3,9){\line(-1,-1)2}
\put(3,9){\line(1,-1)2}
\put(3,11){\line(0,-1)4}
\put(1,5){\line(0,1)2}
\put(5,5){\line(0,1)2}
\put(2.9,.3){$\emptyset$}
\put(3.4,2.85){$\{0\}$}
\put(-.2,4.85){$\{a\}$}
\put(3.3,4.85){$\{b\}$}
\put(5.3,4.85){$\{c\}$}
\put(-.7,6.85){$\{a,b\}$}
\put(3.3,6.85){$\{a,c\}$}
\put(5.3,6.85){$\{b,c\}$}
\put(3.4,8.85){$\{a,b,c\}$}
\put(2.5,11.4){$\{1\}$}
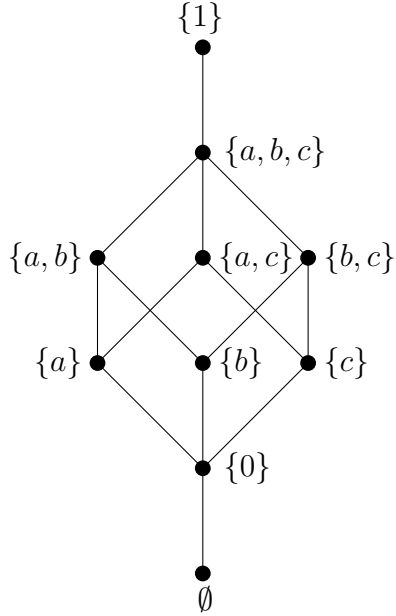
\put(-1.3,-.75){{\rm Figure~5. The poset $\big(A(\mathbf M_3),\le\big)$}}
\end{picture}
\end{center}

\vspace*{4mm}

In fact, $\big(A(\mathbf M_3),\le\big)$ and hence also $\Big(\big(A(\mathbf M_3)\big)^{M_3},\le\Big)$ is a lattice, the latter being denoted by $\Big(\big(A(\mathbf M_3)\big)^{M_3},\vee,\wedge\Big)$. Consider the following mappings which are elements of $\big(A(\mathbf M_3)\big)^{M_3}$:
\[
\begin{array}{c|c|c|c|c}
x & f_a(x) & f_b(x) & f_1(x) &  f(x) \\
\hline
0 & \{0\}  & \{0\}  & \{0\}  &  \{0\}  \\
a & \{a\}  & \{0\}  & \{a\}  &  \{a\}  \\
b & \{0\}  & \{b\}  & \{b\}  &  \{b\}  \\
c & \{0\}  & \{0\}  & \{c\}  &  \{0\}  \\
1 & \{a\}  & \{b\}  & \{1\}  & \{a,b\} 
\end{array}
\]
Then $a\mapsto f_a$ is not a homomorphism from the lattice $(M_3,\vee,\wedge)$ to the lattice $\Big(\big(A(\mathbf M_3)\big)^{M_3},\vee,\wedge\Big)$ since $f_{a\vee b}=f_1\ne f=f_a\vee f_b$.
\end{example}








Authors' address:

Ivan Chajda \\
Palack\'y University Olomouc \\
Faculty of Science \\
Department of Algebra and Geometry \\
17.\ listopadu 12 \\
771 46 Olomouc \\
Czech Republic \\
ivan.chajda@upol.cz

Helmut L\"anger \\
TU Wien \\
Faculty of Mathematics and Geoinformation \\
Institute of Discrete Mathematics and Geometry \\
Wiedner Hauptstra\ss e 8-10 \\
1040 Vienna \\
Austria, and \\
Palack\'y University Olomouc \\
Faculty of Science \\
Department of Algebra and Geometry \\
17.\ listopadu 12 \\
771 46 Olomouc \\
Czech Republic \\
helmut.laenger@tuwien.ac.at
\end{document}